\documentclass[12pt]{article}
\usepackage{graphicx}%
\usepackage{multirow}%
\usepackage{amsmath,amssymb,amsfonts}%
\usepackage{amsthm}%
\usepackage{mathrsfs}%
\usepackage[title]{appendix}%
\usepackage{xcolor}%
\usepackage{textcomp}%
\usepackage{manyfoot}%
\usepackage{booktabs}%
\usepackage{algorithm}%
\usepackage{algorithmicx}%
\usepackage{algpseudocode}%
\usepackage{listings}%
\usepackage{amsmath, amssymb}
\usepackage{graphicx}
\usepackage{hyperref}
\newtheorem{thm}{Theorem}
\newtheorem{prop}[thm]{Proposition}%
 \newtheorem{cor}[thm]{Corollary}

 \theoremstyle{definition}
 \newtheorem{defn}[thm]{Definition}
 \theoremstyle{remark}
 \newtheorem{rem}[thm]{Remark}
 \newtheorem*{ex}{Example}
 \numberwithin{equation}{section}
\usepackage{enumerate}
\usepackage[utf8]{inputenc}
\title{On the connections between $F$-contractions and Meir-Keeler contractions}

\author{Laura Manolescu \\
       Department of Mathematics, Politehnica Universiy of Timi\c soara \\
       \texttt{laura.manolescu@upt.ro} 
        \and
          Pa\c sc G\u avru\c ta \\
        Department of Mathematics, Politehnica Universiy of Timi\c soara\\
      \texttt{pgavruta@gmail.com}}

\date{}

\begin{document}

\maketitle

\begin{abstract}
In this paper, we resolve an open problem concerning the connection between $F$-contractions (Wardowski contractions) and Meir-Keeler contractions. We prove that for a nondecreasing function $F$ that has a point of discontinuity on the right, there exists an $F-$contraction that is not a Meir-Keeler contraction. Moreover, we give a condition on nonlinear $(\varphi,F)$-contractions to be Meir-Keeler contractions. In the final part of the paper, we give a class of $(E,F)-$contractions (in the sense of the paper arXiv:2009.13157 [math.FA] 28 Sep 2020) that are Meir-Keeler contractions.
\end{abstract}

\section{Introduction and preliminaries }\label{sec1}
Let $(X,d)$ be a metric space and $T:X\rightarrow X$ be a mapping. The well-known result of Banach states that every self mapping $T$ defined on a complete metric space which satisfies the Banach contraction inequality, i.e.
$$(\exists) \lambda\in[0,1): d(Tx,Ty)\leq\lambda d(x,y),~(\forall)x,y\in X$$
is a Picard operator ($T$ has a unique fixed $u$ and for every $x\in X,$ the sequence
$\{T^nx\}$ converges to $u$). For a simple proof, see \cite{Palais}. This fundamental result, known as Banach fixed point theorem or Banach contraction principle was generalized in various ways. We mention here some of these generalizations.\\
 \begin{defn} We say that  $T$ is a Meir-Keeler contraction if given  $\varepsilon>0,$ there exists $\delta>0$ such that
 $$(\forall)~x,y\in X,~\varepsilon\leq d(x,y)<\varepsilon+\delta \Longrightarrow~d(Tx, Ty)<\varepsilon.$$
 \end{defn}
 Meir and Keeler \cite{Meir} prove the following theorem. 
\begin{thm}(Meir, Keeler \cite{Meir}) Let $(X,d)$ be a complete metric space and $T$ be a Meir-Keeler contraction. Then $T$ is a Picard operator.
\end{thm}

In 2012, D. Wardowski \cite{Wardowski} introduced a new type of contraction mappings named $F-$ contractions (or Wardowski contractions). We denote by $\mathcal{F}$ the set of all functions $$F:(0,\infty)\rightarrow\mathbb{R}$$ which satisfies  the following conditions

\begin{enumerate}[$(F_1)$]
    \item $F$ is strictly increasing;
    \item for each sequence $\{t_n\}_{n\in\mathbb{N}}\subset(0,\infty),$ $\displaystyle \lim_{n\rightarrow\infty}t_n=0$ iff  $\displaystyle \lim_{n\rightarrow\infty}F(t_n)=-\infty $
    \item there exists $k\in(0,1)$ such that $\displaystyle \lim_{t\rightarrow 0_+}t^kF(t)=0,~\textrm{for any}~ t\in(0,\infty).$
\end{enumerate}
Examples of functions $F,$ which verifies conditions $(F_1)-(F_3):$
\begin{enumerate}[$1)$]
\item $F:(0,\infty)\rightarrow\mathbb{R}, F(t)=\ln t$
\item $F:(0,\infty)\rightarrow\mathbb{R}, F(t)=\ln t+t$
\item $F:(0,\infty)\rightarrow\mathbb{R}, F(t)=-\frac{1}{t^{1/n}}, n\in\mathbb{N}.$
\end{enumerate}

\begin{defn} (Wardowski, \cite{Wardowski}) Let $(X,d)$ be a metric space and $T:X\rightarrow X$ be a mapping. We say that $T$ is a $F-$contraction relative to $\mathcal{F}$ if there exists $\tau>0$ and $F\in\mathcal{F}$ such that the following inequality holds \begin{equation}\label{defFcontr}\tau+F(d(Tx,Ty))\leq F(d(x,y)),\end{equation}
for all $x,y\in X,$ with $d(Tx,Ty)>0.$
\end{defn}
From $(F_1)$ and (\ref{defFcontr}), it follows that a $F-$contraction is a contractive mapping, that is:
$$d(Tx,Ty)<d(x,y),\textrm{for}~x,y\in X,~x\neq y.$$
 In \cite{Wardowski}, Wardowski prove the following result.
 
\begin{thm}(Wardowski \cite{Wardowski}) Let $(X,d)$ be a complete metric space and $T$ be a  $F-$contraction relative to $\mathcal{F}$. Then $T$ is a Picard operator.
\end{thm}

For interesting results concerning Wardowski contractions and some connections with other type of contractions see the papers \cite{Piri}, \cite{Suzuki}, \cite{Turinici2}. See also the survey \cite{Karapinar}.

In the paper \cite{Cvetkovic}, the author give an example of a complete metric space and an example of Meir-Keeler contraction that is not a Wardowski contraction. Also, the author try to prove that on any complete metric space, any Wardowski contraction is a Meir-Keeler contraction. We correct the result from \cite{Cvetkovic}, by giving an example of a Wardowski contraction that is not Meir-Keeler contraction. 

Also, in this paper, we give a condition on nonlinear $(\varphi,F)$-contractions to be Meir-Keeler contractions and we give a class of $(E,F)-$contraction that are Meir-Keeler contractions.

\section{A class of Wardowski contractions that are not Meir-Keeler contractions}

 In this section, we construct a class of a Wardowski contractions that are not Meir-Keeler contractions.
More precisely, we prove the following result.\\

\begin{thm}\label{ourfirstmainthm} Let be $F:(0,\infty)\rightarrow\mathbb{R}$ be a nondecreasing function that has a point of discontinuity on the right. Then there is an $F-$contraction that is not a Meir-Keeler contraction.
\end{thm}
\begin{proof} We denote by $t_0\in(0,\infty)$ a point of discontinuity on the right for $F.$ We have the following properties
$$\tau:=F(t_0+0)-F(t_0)>0$$
and $$F(t)\geq F(t_0+0),~t>t_0.$$
Hence, for $t>t_0,$ we have
\begin{equation}\label{eeq1}
F(t)-F(t_0)\geq F(t_0+0)-F(t_0)=\tau.
\end{equation}
We construct a metric space $(X,d)$ and define an operator $T:X\rightarrow X$
such that $T$ is a $F-$contraction, but is not a Meir-Keeler contraction.
We denote $$A=\{0\}\cup\{t_0\}\cup \{kn\}_{n\geq 1}$$
$$B=\{km+t_0+\gamma_m\}_{m\geq 1}$$
where $k=[2t_0]+1$ and $(\gamma_m)$ is a sequence with the following properties
\begin{equation}\label{eq2}
0<\gamma_m<k-2t_0,~\textrm{for all}~m\geq 1
\end{equation}
\begin{equation}\label{eq3}
\gamma_m\rightarrow 0,~\textrm{as}~ m\rightarrow\infty
\end{equation}
We consider metric space $(X,d)$, where $X=A\cup B$ and $d$ is the usual metric on $\mathbb{R}$. Also, we define the operator $T:X\rightarrow X,$
$$Tx=\begin{cases}
0,~x\in A\\
t_0,~x\in B
\end{cases}$$
We prove that $T$ is $F-$contraction, i.e.
\begin{equation}\label{Fcontr}
Tx\neq Ty\Longrightarrow \tau+F(d(Tx, Ty))\leq F(d(x,y)).
\end{equation}
We denote $x_n=kn,$ $n\geq 1$ and $y_m=km+t_0+\gamma_m,$ $m\geq 1.$
We have $$d(y_n,x_n)=t_0+\gamma_n>t_0$$
$$d(y_m,0)=t_m+t_0+\gamma_m>t_0$$
$$d(y_m,t_0)=km+\gamma_m>k>2t_0>t_0$$
We evaluate $d(x_n,y_m).$ For $m\leq n-1,$ we have
$$d(x_n,y_m)\geq k-t_0-\gamma_m>t_0~\textrm{because}~t_m<k-2t_0.$$
If $m\geq n+1,$ we have $$d(y_m,x_n)\geq k+t_0+\gamma_m>t_0.$$
Hence, if $Tx\neq Ty,$ it follows $x\in A,$ $y\in B$ and then because $d(y_m,x_n)>t_0,$ from (\ref{eeq1}), we have
$$F(d(y_m,x_n))-F(t_0)\geq\tau.$$  Hence
$$F(d(Tx_n, Ty_m))+\tau\leq F((d(y_n,x_n)))$$
so $T$ is $F-$contraction.

We show now that $T$ is not Meir-Keeler contraction. We observe that $T$ does not fulfill the condition
$$(\forall) \varepsilon>0, (\exists) \delta>0: \varepsilon\leq d(x,y)<\varepsilon+\delta\Longrightarrow d(Tx,Ty)<\varepsilon,$$
for $\varepsilon=t_0>0.$ Indeed, $$(\forall) \delta>0~\textrm{we have}~ t_0<d(x_n,y_n)<t_0+\delta,~\textrm{for a sufficient large n}$$
because $d(x_n,y_n)=t_0+\gamma_n\rightarrow t_0,~\textrm{as}~n\rightarrow\infty.$
But $d(Tx_n, Ty_n)=t_0\geq t_0.$
\end{proof}
\begin{rem} In the paper \cite{Manolescu}, we raised the question if a $F-$contraction relative to nondecreasing function is Meir-Keeler contraction. In the Theorem \ref{ourfirstmainthm}, we
respond also to this problem.
\end{rem}
\section{A class of nonlinear $F-$contractions that are Meir-Keeler contractions}
In 2018, D. Wardowski introduced a new type of $F-$contraction, named nonlinear $F-$contraction (or $(\varphi,F)-$contraction) and give conditions for a $(\varphi,F)-$contraction to be a Picard operator.

\begin{defn}(Wardowski \cite{Wardowski1})
Let be the functions $F:(0,\infty)\rightarrow\mathbb{R}$ and $\varphi:(0,\infty)\rightarrow (0,\infty).$ Let $(X,d)$ be a  metric space and $T:X\rightarrow X$. We say that $T$ is a nonlinear $F-$contraction (or $(\varphi,F)$-contraction) if
\begin{equation}\label{eqW}
   (\forall)~x,y\in X, Tx\neq Ty\Longrightarrow \varphi(d(x,y))+
   F(d(Tx, Ty))\leq F(d(x,y)).  
\end{equation}
\end{defn}
\begin{thm} (Wardowski \cite{Wardowski1})\label{Wardowski1}
Let $(X,d)$ be a complete metric space and $T:X\rightarrow X$ be a $(\varphi,F)$-contraction. We suppose that 
\begin{enumerate}[$(i)$]
    \item $F$ is strictly increasing;
    \item $\displaystyle \lim_{t\rightarrow 0^+}F(t)=-\infty;$
    \item $\displaystyle \liminf_{s\rightarrow t^+}\varphi(s)>0,~\textrm{for all}~ t\geq 0.$
\end{enumerate}
Then $T$ is a Picard operator.
\end{thm}
D. Wardowski \cite{Wardowski1} obtained applications of this result to the following equation of Volterra type
$$x(t)=\int_0^tK(t,s,x(s))ds+h(t)$$ and afterwards to the equation $$\displaystyle V(t,x(t))=S(t,\int_0^tH(t,s,x(s)))ds.$$

 In the paper \cite{Manolescu}, we raised the following problem: under what conditions $(\varphi,F)-$contractions are Meir-Keeler contractions? In this sense, we obtained the following result.

\begin{thm} (Manolescu, G\u avru\c ta, Khojasteh \cite{Manolescu})\label{Manolescu}
Let $(X,d)$ be a complete metric space and $T:X\rightarrow X$ be a $(\varphi,F)$-contraction. We suppose that $F$ is nondecreasing and $F$ is continuous at right. If $\varphi$ is lower semi-continuous, then $T$ is a Meir-Keeler contraction.
\end{thm}

In the next Theorem, we give a more generalized result.

\begin{thm} Let be $F:(0,\infty)\rightarrow\mathbb{R}$ a nondecreasing mapping and $\varphi:(0,\infty)\rightarrow (0,\infty)$ be such that 
\begin{equation*}
\limsup_{\substack{t_n\rightarrow t\\ t_n>t}}\varphi(t_n)>F(t^+)-F(t), t\in(0,\infty).
\end{equation*}
Let $(X,d)$ be a complete metric space and $T$ be a $(\varphi, F)$-contraction on $X.$
Then $T$ is Meir-Keeler contraction.
\end{thm}
\begin{proof} We suppose that $T$ is not a Meir-Keeler contraction, i.e.
$$(\exists) \varepsilon_0~\textrm{such that}~(\forall)\delta>0,~(\exists)x_{\delta}, y_{\delta}\in X~\textrm{such that}~ 
\varepsilon_0\leq d(x_{\delta},y_{\delta})<\varepsilon_0+\delta$$
and $$d(Tx_{\delta},Ty_{\delta})\geq\varepsilon_0.$$
We take here $\delta=2^{-n}, n\in\mathbb{N}.$ It follows that there exists two sequences $\{x_n\},\{y_n\}\subset X$ such that
$$\varepsilon_0\leq d(x_n,y_n)<\varepsilon_0+2^{-n}$$
and $$d(Tx_n, Ty_n)\geq \varepsilon_0, n\in\mathbb{N}.$$
From the above inequality it follows $d(Tx_n, Ty_n)\neq 0,$ hence
$$\varphi(d(x_n,y_n))+F(d(Tx_n,Ty_n))\leq F(d(x_n,y_n))$$ and, since $F$ is nondecreasing, we have 
$$F(d(Tx_n,Ty_n))\geq F(\varepsilon_0).$$
We deduce $\varphi(d(x_n,y_n))\leq F(d(x_n,y_n))-F(\varepsilon_0).$ We observe that $$d(x_n,y_n)>\varepsilon_0,~(\forall) n.$$
Indeed, if there exists $n,$ with $d(x_n, y_n)=\varepsilon_0,$ it follows $$\varphi(\varepsilon_0)\leq 0,$$ contradiction with the hypothesis. Then, with $t_n:=d(x_n,y_n),~n\in\mathbb{N},$
we deduce that $$\limsup_{\substack{t_n\rightarrow \varepsilon_0\\t_n>\varepsilon_0}}\varphi(t_n)\leq F(\varepsilon_0^+)-F(\varepsilon_0),$$ contradiction with the hypothesis.
\end{proof}
\section{A class of $(E,F)$-contractions that are Meir-Keeler contractions}

 In the paper \cite{Gavruta1}, we introduced the class of $(E,F)$-contractions and considered for the first time a condition that does not refer to the monotony of any of the functions. In this section, we construct a class of $(E,F)$-contractions that are Meir-Keeler contractions.

Let be the functions $E,F:(0,\infty)\rightarrow\mathbb{R}.$ We say that $(E,F)$ satisfies condition $(C_1)$ if
$$\textrm{for}~t,s\in(0,\infty),~t\leq s\Longrightarrow E(t)<F(s).$$
\begin{prop}\label{propEF}
1. We suppose that the pair $(E,F)$ satisfies condition $(C_1).$ Then $E(t)<F(t), t>0.$\\
2. If $E(t)<F(t),$ $t>0$ and $F$ is nondecreasing, then $(E,F)$ satisfies condition $(C_1).$
\end{prop}
We give an example of functions $(E,F)$ that satisfies condition $(C_1),$ but none of them is not monotone.
\begin{ex} Let be $F:(0,\infty)\rightarrow\mathbb{R},$
$$F(t)=\begin{cases} \dfrac{5}{2}, 0<t<\dfrac{1}{2}\\
\\
\dfrac{1+t^2}{t}, t\geq\frac{1}{2}
\end{cases}$$
We take $E=\lambda F,$ where $0<\lambda<\dfrac{4}{5}.$
\end{ex}

\begin{defn} We suppose that $(E,F)$ satisfies the condition $(C_1).$ We say that $T:X\rightarrow X$ is $(E,F)$-contraction if the following condition holds
\begin{equation}\label{defEF}
Tx\neq Ty\Longrightarrow F(d(Tx,Ty))\leq E(d(x,y)).
\end{equation}
\end{defn}

\begin{prop}\label{propEF2} If $T:X\rightarrow X$ is $(E,F)$-contraction, then $T$ is contractive.
\end{prop}
\begin{proof} We have to prove that if $x\neq y$ it follows that $d(Tx,Ty)<d(x,y).$\\
If $Tx=Ty,$ it is clear.\\
If $Tx\neq Ty,$ we suppose $d(x,y)\leq d(Tx, Ty).$ From condition $(C_1),$ it follows
$$E(d(x,y))<F(d(Tx, Ty)),$$ contradiction with relation (\ref{defEF}).
\end{proof}

\begin{thm} Let be $E,F:(0,\infty)\rightarrow\mathbb{R}$ such that
the following condition holds
$$(C_1)~t\leq s\Longrightarrow E(t)<F(s).$$
We suppose that 
\begin{enumerate}[$i)$]
\item $F$ is nondecreasing and $\displaystyle\liminf_{\substack{t_n\rightarrow t\\t_n>t}}E(t_n)<F(t), (\forall) t>0$ or
\item $F$ is continuous on the right and $\displaystyle\liminf_{\substack{t_n\rightarrow t\\t_n\geq t}}E(t_n)<F(t), (\forall) t>0$
\end{enumerate}
If $T$ is $(E,F)$-contraction on the complete metric space $(X,d)$, then $T$ is a Meir-Keeler contraction.
\end{thm}
\begin{proof} We suppose that $T$ is not a Meir-Keeler contraction, hence 
there exists $\varepsilon_0$ and two sequences $(x_n), (y_n)\subset X$ such that
\begin{equation*}
\varepsilon_0\leq d(x_n,y_n)<\varepsilon_0+2^{-n}
\end{equation*}
and \begin{equation}\label{eqdem}
d(Tx_n,T y_n)\geq\varepsilon_0,~n\in\mathbb{N}.
\end{equation}
From the above inequality, it follows that $d(Tx_n,Ty_n)\neq 0$ and so 
\begin{equation}\label{steluta}
F(d(Tx_n,Ty_n))\leq E(d(x_n,y_n)),~n\in\mathbb{N}.
\end{equation}
If $F$ is nondecreasing, from (\ref{eqdem}) and (\ref{steluta}), we have
\begin{equation}\label{dublusteluta}
F(\varepsilon_0)\leq F(d(Tx_n,Ty_n))\leq E(d(x_n,y_n)).
\end{equation}
If there exists $n$, with $d(x_n,y_n)=\varepsilon_0,$ from (\ref{dublusteluta}), we obtain
$$F(\varepsilon_0)\leq E(\varepsilon_0),$$
contradiction with the fact that $E(t)<F(t),$ $t>0,$ which follows from $(C_1).$ If we define $t_n:=d(x_n,y_n)$, then $$t_n>\varepsilon_0, n\in\mathbb{N},~t_n\rightarrow\varepsilon_0.$$
From (\ref{dublusteluta}), it follows that
$$F(\varepsilon_0)\leq\liminf_{n\rightarrow\infty} E(t_n),$$ contradiction with the hypothesis.\\
\\
We consider now that $F$ is continuous on the right. From Proposition \ref{propEF2}, it follows that $T$ is contractive, hence 
\begin{equation}\label{eqEF}
\varepsilon_0\leq d(Tx_n, T y_n)<d(x_n,y_n).\end{equation}
We denote by $t_n:=d(x_n,y_n)$ and $s_n:=d(Tx_n,Ty_n), ~n\in\mathbb{N}.$\\
We have $t_n\rightarrow\varepsilon_0,~t_n\geq\varepsilon_0,~n\in\mathbb{N}$
and from (\ref{eqEF}), it follows that $$s_n\geq\varepsilon_0, n\in\mathbb{N}~\textrm{and}~ s_n\rightarrow\varepsilon_0.$$
From (\ref{steluta}) it follows that $F(s_n)\leq E(t_n),~n\in\mathbb{N},$ hence $$\liminf_{n\rightarrow\infty} F(s_n)\leq\liminf_{n\rightarrow\infty} E(t_n).$$
Since $F$ is continuous on the right, it follows that $\displaystyle F(\varepsilon_0)\leq\liminf_{n\rightarrow\infty} F(t_n).$
 contradiction with the hypothesis.

\end{proof}

\begin{cor} Let be the functions $E,F:(0,\infty)\rightarrow\mathbb{R}$ such that
\begin{enumerate}[$1)$]
    \item $E(t)<F(t)$, for all $t>0;$
    \item $F$ is nondecreasing;
    \item $\displaystyle\liminf_{n\rightarrow\infty}E(t_n)<F(\varepsilon),$ $(\forall)\varepsilon>0$ and $(\forall)~ (t_n)\subset(\varepsilon,\infty)$
\end{enumerate}

Let $(X,d)$ be a complete metric space and $T:X\rightarrow X$ be a $(E,F)$-contraction. Then $T$ is a Meir-Keeler contraction.
\end{cor}

\section{Conclusion}\label{sec13}
We give an example of a $F-$contraction that is not Meir-Keeler contraction, and so we correct a result from the paper \cite{Cvetkovic}. Also, we obtain new conditions for a class of nonlinear $F-$contractions (in the sense of the paper \cite{Wardowski1}) and for a class of $(E,F)-$ contractions (in the sense of the paper \cite{Gavruta1}) to be Meir-Keeler contractions.











\end{document}